\documentclass{amsart}

\newtheorem{theorem}{Theorem}[section]
\newtheorem{lemma}[theorem]{Lemma}
\newtheorem{proposition}[theorem]{Proposition}
\newtheorem{corollary}[theorem]{Corollary}

\theoremstyle{definition}
\newtheorem{definition}[theorem]{Definition}

\theoremstyle{remark}
\newtheorem{remark}[theorem]{Remark}

\numberwithin{equation}{section}

\usepackage{amsfonts,amsmath,amssymb}
\usepackage{mathrsfs,mathtools,stmaryrd,wasysym}
\usepackage{enumerate}
\usepackage{xspace} 
\usepackage{esint}
\usepackage{hyperref}
\DeclareMathAlphabet{\mathpzc}{OT1}{pzc}{m}{it}

\newcommand{\R}{\mathbb{R}}

\newcommand{\calK}{\mathcal{K}}

\newcommand{\bv}{\mathbf{v}}

\newcommand{\GRAD}{\nabla}
\DeclareMathOperator{\DIV}{div}
\newcommand{\diff}{\, \mbox{\rm d}}
\newcommand{\ie}{i.e.,\@\xspace}

\DeclareMathOperator{\diam}{diam}
\newcommand{\calZ}{\mathcal{Z}}
\newcommand{\bF}{{\mathbf{F}}}
\newcommand{\bL}{{\mathbf{L}}}
\newcommand{\bW}{{\mathbf{W}}}
\newcommand{\bH}{{\mathbf{H}}}
\newcommand{\bG}{{\mathbf{G}}}
\newcommand{\bE}{{\mathbf{E}}}
\newcommand{\bef}{{\mathbf{f}}}
\newcommand{\bX}{{\mathbf{X}}}
\newcommand{\bvarphi}{{\boldsymbol{\varphi}}}

\newcommand{\btau}{{\boldsymbol{\tau}}}
\newcommand{\bg}{{\mathbf{g}}}
\newcommand{\bw}{{\mathbf{w}}}
\newcommand{\T}{\mathscr{T}_h}
\newcommand{\V}{\mathbf{V}}
\renewcommand{\P}{\mathcal{P}}
\newcommand{\vare}{\varepsilon}

\newcommand{\Tr}{\mathbb{T}}
\newcommand{\pe}{{\mathsf{\pi}}}
\newcommand{\ue}{{\mathsf{u}}}

\DeclareMathOperator{\signum}{sign}

\newcommand{\dist}[1]{ {\textup{\textsf{d}}}_{#1} }

\newcommand{\Edges}{\mathscr{E}}
\newcommand{\Vertices}{\mathscr{V}}
\newcommand{\Midpoints}{\mathscr{M}}
\newcommand{\textv}{{\texttt{v}}}
\newcommand{\texte}{{\texttt{e}}}
\newcommand{\textm}{{\texttt{m}}}

\begin{document}

\title[Stability of the Stokes projection on weighted spaces]{Stability of the Stokes projection on weighted spaces and applications}
\author{Ricardo G.~Dur\'an}
\address{IMAS (UBA-CONICET) and Departamento de Matem\'atica, Facultad de Ciencias Exactas y Naturales, Universidad de Buenos Aires, Buenos Aires, Argentina.}
\curraddr{}
\email{rduran@dm.uba.ar}
\thanks{The first author was partially supported by ANPCyT grant PICT 2014-1771, by CONICET grant
11220130100006CO, and by Universidad de Buenos Aires grant 20020120100050BA}

\author{Enrique Ot\'arola}
\address{Departamento de Matem\'atica, Universidad T\'ecnica Federico Santa Mar\'ia, Valpara\'iso, Chile.}
\curraddr{}
\email{enrique.otarola@usm.cl}
\thanks{The second author was partially supported by FONDECYT grant 11180193.}

\author{Abner J.~Salgado}
\address{epartment of Mathematics, University of Tennessee, Knoxville, TN 37996, USA.}
\curraddr{}
\email{asalgad1@utk.edu}
\thanks{The third author was partially supported by NSF grant DMS-1720213.}

\date{}
\dedicatory{}

\begin{abstract}
We show that, on convex polytopes and two or three dimensions, the finite element Stokes projection is stable on weighted spaces $\bW^{1,p}_0(\omega,\Omega) \times L^p(\omega,\Omega)$, where the weight belongs to a certain Muckenhoupt class and the integrability index can be different from two. We show how this estimate can be applied to obtain error estimates for approximations of the solution to the Stokes problem with singular sources.
\end{abstract}

\maketitle

\section{Introduction}
\label{sec:intro}

In this work we shall be interested in the stability and approximation properties of the finite element Stokes projection when measured over weighted norms. To be precise, let $d \in \{2,3\}$ and $\Omega \subset \R^d$ be a convex polytope. Assume that $\Tr = \{\T\}_{h>0}$ is a family of quasiuniform triangulations of $\bar \Omega$ parametrized by their mesh size $h>0$ and $\V_h \times \P_h$ is a pair of finite element spaces constructed over the mesh $\T$. To describe the question that we wish to address here let $(\ue,\pe) \in \bW^{1,1}_0(\Omega) \times L^1(\Omega)/\R$, with $\ue$ solenoidal (see Section~\ref{sec:prelim} for notation), and define $(\ue_h,\pe_h) \in \V_h \times \P_h$ to be its Stokes projection, \ie the pair $(\ue_h,\pe_h)$ is such that
\begin{equation}
\label{eq:Stokesproj}
  \begin{dcases}
    \int_\Omega \left[ \nabla \ue_h : \nabla \bv_h - \pe_h \DIV \bv_h \right] \! \diff x= \int_\Omega \left[ \nabla \ue : \nabla \bv_h - \pe \DIV \bv_h \right] \! \diff x &  \! \forall \bv_h \in \V_h, 
\\
    \int_\Omega q_h \DIV \ue_h \diff x = 0 & \! \forall q_h \in \P_h.
  \end{dcases}
\end{equation}
With this notation, the main result in our work is that, for a certain range of integrability indices $p$ and a certain class of Muckenhoupt weights $\omega$, we have
\begin{equation}
\label{eq:WeightedStability}
  \| \nabla \ue_h \|_{\bL^p(\omega,\Omega)}  +  \| \pe_h \|_{L^p(\omega,\Omega)} \lesssim \| \nabla \ue \|_{\bL^p(\omega,\Omega)}  +  \| \pe \|_{L^p(\omega,\Omega)}.
\end{equation}

Our main motivation for the development of such estimates is the study of the Stokes problem
\begin{equation}
\label{eq:StokesStrong}
  \begin{dcases}
      -\Delta \ue + \GRAD \pe = \bef, & \text{in } \Omega, \\
      \DIV \ue = 0, & \text{in } \Omega, \\
      \ue = 0, & \text{on } \partial\Omega,
  \end{dcases}
\end{equation}
in the case where the forcing term $\bef$ is allowed to be singular. Essentially, by introducing a weight, we can allow for forces such that $\bef \notin \bW^{-1,2}(\Omega)$. In particular, our theory will allow the following particular examples. For a fixed $\bF \in \R^d$ we can set $\bef = \bF \delta_z$, where $\delta_z$ denotes the Dirac delta supported at the interior point $z \in \Omega$. Similarly, if $\Gamma$ denotes a smooth curve or surface without boundary contained in $\Omega$, we can allow the components of $\bef$ to be measures supported in $\Gamma$.

While the stability and approximation properties for the Stokes problem in energy type norms has a rich history and is by now well established, the derivation of these properties in non energy norms is more delicate. To our knowledge, the first works that address these questions in a non energy setting are \cite{MR935076,MR995211}. In these references, the authors establish a $L^\infty$-norm almost stability (up to logarithmic factors) in two dimensions. Later, in view of the weighted a priori estimate for a solution of the divergence operator of \cite{MR1880723}, the results of \cite{MR995211} were extended to three dimensions; see \cite[Section 3]{MR1880723} for a discussion. We would also like to mention reference \cite{MR2217368} for results on domains with smooth boundaries. Results withouth logarithmic factors 
where first established in \cite{MR2121575}, albeit under certain restrictions on the internal angles of the domain. This last assumption was finally removed in \cite{MR3422453}. The state of the art is that, simply put, the Stokes projection is stable in $\bW^{1,p}_0(\Omega) \times L^p(\Omega)/\R$ for $p \in (1,\infty]$ if the domain $\Omega$ is a convex polytope.

We must remark that, in the PDE literature, the idea of introducing weights to handle singular sources is by now well established. There is a vast amount of literature dealing with weighted a priori estimates for solutions of elliptic equations and systems, 
and for models of incompressible fluids that are even more general than \eqref{eq:StokesStrong}; see for instance \cite{MR3582412}. However, in most of these works, it is usually assumed that the domain is at least $C^1$, which is not finite element friendly. Two exceptions are \cite{DDO:17,OS:17infsup}. In \cite{DDO:17} the well posedness of the Poisson problem in $W^{1,p}_0(\omega,\Omega)$ is established for all $p \in (1,\infty)$ and $\omega \in A_p$, provided $\Omega$ is a convex polytope. In addition, the stability of the Ritz projection is obtained for $p \in [2,\infty)$ and $\omega \in A_1$, and for $p =2$ and $\omega^{-1} \in A_1$. On the other hand, \cite{OS:17infsup} works on general Lipschitz domains, and shows that the Poisson and Stokes problems are well posed, provided $p$, that depends on the domain, is restricted to a neighborhood of $2$ and the weight is regular near the boundary ($\omega \in A_p(\Omega)$ in the notation of that work).

From the discussion given above, it is clear that the stability of the Stokes projection is open and, in light of applications, needed. This is the main contribution of our work.

Our presentation will be organized as follows. We set notation in Section \ref{sec:prelim}, where we also recall the definition of Muckenhoupt weights and introduce the weighted spaces we shall work with. In addition, in Section \ref{sec:Stokes}, we introduce a saddle point formulation of the Stokes problem \eqref{eq:StokesStrong} in weighted spaces and review well-posedness results. In Section \ref{sec:finite_element_approx} we introduce the discrete setting in which we will operate. Section \ref{sec:discrete_stability} is dedicated to obtaining the stability of the finite element Stokes projection in weighted spaces; this is one of the highlights of our work. As an immediate application, Section \ref{sec:error_estimates} studies the development of $\bL^p$--error estimates for the error approximation of the velocity field. We also specialize these results and study the approximation of the Stokes problem with a forcing term that is a linear combination of Dirac measures.
All the developments of the previous sections rest on a series of assumptions on the finite element velocity--pressure pairs. For this reason in, the final, Section \ref{sec:Pairs} we derive a continuous weighted inf--sup condition and study some suitable finite element pairs that satisfy all the assumptions that our theory rests upon.

\section{Notation and preliminaries}
\label{sec:prelim}

We begin by fixing notation and the setting in which we will operate. Throughout this work $d \in \{2,3\}$ and $\Omega \subset \R^d$ is an open, bounded, and convex polytope. If $\mathcal{W}$ and $\mathcal{Z}$ are Banach function spaces, we write $\mathcal{W} \hookrightarrow \mathcal{Z}$ to denote that $\mathcal{W}$ is continuously embedded in $\mathcal{Z}$. We denote by $\mathcal{W}'$ and $\|\cdot\|_{\mathcal{W}}$ the dual and the norm of $\mathcal{W}$, respectively.

For $E \subset \Omega$ open and $f : E \to \R$, we set
\[
 \fint_E f \diff x  = \frac{1}{|E|}\int_{E} f \diff x.
\]
For $w \in L^1_{\mathrm{loc}}(\Omega)$, the Hardy--Littlewood maximal operator is defined by
\begin{equation}
\label{eq:Maximal}
 \mathcal{M} w(x) = \sup_{Q \ni x} \fint_{Q} |w(y)| \diff y,
\end{equation}
where the supremum is taken over all cubes $Q$ containing $x$.

Given $p \in (1,\infty)$, we denote by $p'$ its H\"older conjugate, \ie the real number such that $1/p + 1/p' = 1$. By $a \lesssim b$ we will denote that $a \leq C b$, for a constant $C$ that does not depend on $a$, $b$ nor the discretization parameters. The value of $C$ might change at each occurrence.

\subsection{Weights and weighted Sobolev spaces}
\label{sub:weights}

By a weight we mean a locally integrable, nonnegative function defined on $\R^d$. If $\omega$ is a weight and $E \subset \R^d$ we set
\[
  \omega(E) = \int_E \omega \diff x.
\]
Of particular interest to us will be the so--called Muckenhoupt $A_p$ weights \cite{MR1800316,MR0293384,MR1774162}. 

\begin{definition}[Muckenhoupt class $A_p$]
Let $p \in [1,\infty)$ we say that a weight $\omega \in A_p$ if
\begin{equation}
\begin{aligned}
\label{A_pclass}
\left[ \omega \right]_{A_p} & := \sup_{B} \left( \fint_{B} \omega \diff x\right) \left( \fint_{B} \omega^{1/(1-p)} \diff x \right)^{p-1}  < \infty, \quad p \in (1,\infty),
\\
\left[ \omega \right]_{A_1} & := \sup_{B} \left( \fint_{B} \omega \diff x\right)  \sup_{x \in B} \frac{1}{\omega(x)}< \infty, \quad p =1,
\end{aligned}
\end{equation}
where the supremum is taken over all balls $B$ in $\R^d$. In addition, $A_{\infty} := \bigcup_{p > 1} A_p$.
We call $[\omega]_{A_p}$, for $p \in [1,\infty)$, the Muckenhoupt characteristic of $\omega$.
\end{definition}

Notice that there is a certain symmetry in the $A_p$ classes with respect to H\"older conjugate exponents. If $\omega \in A_p$, then its conjugate $\omega':= \omega^{1/(1-p)} \in A_{p'}$ and 
\[
[ \omega' ]_{A_{p'}} = [ \omega ]^{1/(p-1)}_{A_p}.
\]
We comment also that, following \cite[Chapter 7.1]{MR1800316}, an equivalent characterization of $\omega \in A_1$ is that for almost every $x$,
\begin{equation}
\label{eq:A1usingMax}
   \mathcal{M} \omega(x) \lesssim \omega(x).
\end{equation}
The class of $A_p$ weights was introduced by Muckenhoupt in \cite{MR0293384} where he showed that the $A_p$ weights are precisely those weights for which the Hardy-Littlewood maximal operator is bounded on weighted Lebesgue spaces; see \cite{MR0293384} and \cite[Theorem 7.3]{MR1800316}.

Distances to lower dimensional objects are prototypical examples of Muckenhoupt weights. In particular, if $\calK \subset \Omega$ is a smooth compact submanifold of dimension $k \in [0,d) \cap \mathbb{Z}$ then, owing to \cite{MR3215609} and \cite[Lemma 2.3(vi)]{MR1601373}, we have that the function
\begin{equation}
\label{distance_A2}
  \dist{\calK}^\alpha(x) = \textup{dist}(x,\calK)^{\alpha}
\end{equation}
belongs to the class $A_p$ provided
\[
  \alpha \in \left( -(d-k), (d-k)(p-1) \right).
\]
This allows us to identify three particular cases:

\begin{enumerate}[(i)]
\item Let $d>1$ and $z \in \Omega$, then the weight $\dist{z}^\alpha \in A_2$ if and only if $ \alpha \in (- d , d) $.
\label{item1}
  
\item Let $d \geq 2$ and $\gamma \subset \Omega$ be a smooth closed curve without self intersections. We have that $\dist{\gamma}^\alpha \in A_2$ if and only if $\alpha \in \left( -(d-1) , d-1 \right) $.
\label{item2}

\item Finally, if $d = 3$ and $\Gamma \subset \Omega$ is a smooth surface without boundary, then $\dist{\Gamma}^\alpha \in A_2$ if and only if $\alpha \in (-1,1)$.
\label{item3}

\end{enumerate}

It is important to notice, first, that in all the examples shown above we have that either the weight or its inverse, which is the conjugate within the $A_2$ class, belongs to $A_1$. Second, since the lower dimensional objects are strictly contained in $\Omega$, there is a neighborhood of $\partial\Omega$ where the weight has no degeneracies or singularities. In fact, it is continuous and strictly positive. This observation motivates us to define a restricted class of Muckenhoupt weights that will be of importance for the analysis that follows. The next definition is inspired by \cite[Definition 2.5]{MR1601373}.
\begin{definition}[class $A_p(D)$]
\label{def:ApOmega}
Let $D \subset \R^d$ be a Lipschitz domain. For $p \in (1, \infty)$ we say that $\omega \in A_p$ belongs to $A_p(D)$ if there is an open set $\mathcal{G} \subset D$, and positive constants $\varepsilon>0$ and $\omega_l>0$, such that:
\begin{enumerate}[(a)]
\item $\{ x \in \Omega: \mathrm{dist}(x,\partial D)< \varepsilon\} \subset \mathcal{G}$, 
\item $\omega \in C(\bar {\mathcal{G}})$, and  
\item $\omega_l \leq \omega(x)$ for all $x \in \bar{\mathcal{G}}$. 
\end{enumerate}
\end{definition}

Notice that the weights described in \eqref{item1}--\eqref{item3} belong to the restricted Muckenhoupt class $A_2(\Omega)$. 
The latter has been shown to be crucial in the analysis of  \cite{OS:17infsup} that guarantees the well--posedness of problem \eqref{eq:StokesStrong} in the weighted Sobolev spaces that we define below.

Let $p \in (1,\infty)$, $\omega \in A_p$, and $E \subset \R^d$ be an open set. We define $L^p(\omega,E)$ as the space of Lebesgue $p$--integrable functions with respect to the measure $\omega\diff x$. We also define the weighted Sobolev space $W^{1,p}(\omega,E)$ as the set of functions $v \in L^p(\omega,E)$ with weak derivatives $D^{\alpha} v \in L^p(\omega,E)$ for $|\alpha| \leq 1$. The norm of a function $v \in W^{1,p}(\omega,E)$ is given by
\begin{equation}
\label{eq:norm}
 \| v \|_{W^{1,p}(\omega,E)}:= \left(  \| v \|_{L^p(\omega,E)}^p +  \| \nabla v \|_{L^p(\omega,E)}^p  \right)^{1/p}.
\end{equation}
We also define $W_0^{1,p}(\omega,E)$ as the closure of $C_0^{\infty}(E)$ in $W^{1,p}(\omega,E)$.  It is remarkable that most of the properties of classical Sobolev spaces have a weighted counterpart. This is not because of the specific form of the weight but rather due to the fact that the weight $\omega$ belongs to the Muckenhoupt class $A_p$. If $p \in (1,\infty)$ and $\omega$ belongs to $A_p$, then $L^p(\omega,E)$ and $W^{1,p}(\omega,E)$ are Banach spaces \cite[Proposition 2.1.2]{MR1774162} and smooth functions are dense \cite[Corollary 2.1.6]{MR1774162}; see also \cite[Theorem 1]{MR2491902}. In addition, \cite[Theorem 1.3]{MR643158} guarantees a weighted Poincar\'e inequality which, in turn, implies that over $W^{1,p}_0(\omega,E)$ the seminorm $\| \nabla v \|_{L^p(\omega,E)}$ is an equivalent norm to the one defined in \eqref{eq:norm}. 

Spaces of vector valued functions will be denoted by boldface, that is 
\[
 \mathbf{W}_0^{1,p}(\omega,E) = [ W_0^{1,p}(\omega,E) ]^d, \quad \| \GRAD \bv \|_{\bL^p(\omega,E)}:= \left( \sum_{i=1}^d \| \nabla v^i \|^p_{\bL^p(\omega,E)} \right)^{1/p},
\]
where $\bv = (v^1,\dots,v^d)^\intercal$. 

For future use we recall a particular Sobolev--type embedding theorem between weighted spaces. For the general case we refer to \cite{MR805809,MR1292572,MR2454455} and \cite[Section 6]{NOS3}.

\begin{proposition}[embedding in weighted spaces]
\label{prop:embedding}
Let $p \in (1,\infty)$ and $\omega \in A_p$. Assume that, for all $x \in \Omega$ and $0<r\leq R$, we have that
\[
  \frac{r^{p+d}}{R^{p+d}} \frac{\omega(B(x,R))}{\omega(B(x,r))} \lesssim 1,
\]
then $W^{1,p}(\omega,\Omega) \hookrightarrow L^p(\Omega)$ and $W^{1,p'}(\Omega) \hookrightarrow L^{p'}(\omega',\Omega)$.
\end{proposition}

\subsection{The Stokes problem in weighted spaces}
\label{sec:Stokes}

We begin with a motivation for the use of weights. Let us assume that \eqref{eq:StokesStrong} is posed over the whole space $\R^d$ and that $\bef = \bF \delta_z$ for some $z \in \R^d$. The results of \cite[Section IV.2]{MR2808162} thus provide the following asymptotic behavior of the solution $(\ue, \pe)$ to problem \eqref{eq:StokesStrong} near the point $z$:
\begin{equation}\label{asympt-x0}
| \nabla \ue (x) | \approx |x-z|^{1-d} \quad \textrm{and} \quad | \pe(x) | \approx |x-z|^{1-d},
\end{equation}
so that $|\nabla \ue |, \pe \notin L^2(\R^d)$. However, basic computations reveal that, for every ball $B$,
\[
\alpha \in (d-2,\infty) \implies
\int_B \dist{z}^{\alpha} | \nabla \ue |^2 \diff x< \infty, 
\quad
\int_B \dist{z}^{\alpha} | \pe |^2 \diff x< \infty.
\]
This heuristic suggests to seek solutions to problem \eqref{eq:StokesStrong} in weighted Sobolev spaces \cite{MR3582412,OS:17infsup}. In what follows we will make these considerations rigorous.
 
Let $\omega \in A_p$. Given $\bef \in \bW^{-1,p}(\omega,\Omega)$, we seek for $(\ue,\pe) \in \bW^{1,p}_0(\omega,\Omega) \times L^p(\omega,\Omega)/\R$ such that
\begin{equation}
\label{eq:StokesWeak}
  \begin{dcases}
  a(\ue,\bv) + b(\bv,\pe) =  \langle \bef, \bv \rangle  & \forall \bv \in \bW^{1,p'}_0(\omega',\Omega), \\
  b(\ue,q) = 0 &\forall q \in L^{p'}(\omega',\Omega)/\R,
  \end{dcases}
\end{equation}
where $\langle \cdot, \cdot \rangle$ denotes the duality pairing between $\bW^{-1,p}(\omega,\Omega):=\bW^{1,p'}_0(\omega',\Omega)'$ and $\bW^{1,p'}_0(\omega',\Omega)$. Finally, to shorten notation, here and in what follows, we set
\[
  a(\bv,\bw) = \int_\Omega \GRAD \bv : \GRAD \bw \diff x, \qquad b(\bv,q) = -\int_\Omega q \DIV \bv \diff x.
\]

The well--posedness of \eqref{eq:StokesWeak} in Lipschitz domains was studied in \cite[Theorem 17]{OS:17infsup}. The main result is summarized below.

\begin{proposition}[well--posedness in weighted spaces]
\label{prop:OSwellposed}
Let $d \in \{2,3\}$ and $\Omega \subset \R^d$ be a Lipschitz domain. There exists $\epsilon = \epsilon (d,\Omega) \in (0,1]$ such that if $P = 2+\epsilon$, $p \in (P',P)$, and $\omega\in A_p(\Omega)$, problem \eqref{eq:StokesWeak} is well posed. In other words, for all $\bef \in \bW^{-1,p}(\omega,\Omega)$ problem \eqref{eq:StokesWeak} has a unique solution $(\ue,\pe) \in \bW^{1,p}_0(\omega,\Omega) \times L^p(\omega,\Omega)/\R$ and the following stability estimate holds
\begin{equation}
\label{eq:aprioristokes}
  \| \GRAD \ue \|_{\bL^p(\omega,\Omega)} + \| \pe \|_{L^p(\omega,\Omega)} \lesssim \| \bef \|_{\bW^{-1,p}(\omega,\Omega)}.
\end{equation}
\end{proposition}

\begin{remark}[$p<2$]
Strictly speaking \cite[Theorem 17]{OS:17infsup} only shows well--posedness for $p \geq 2$. However, using the equivalent characterization of well--posedness via inf--sup conditions given in \cite[Theorem 2.1]{MR972452}, see also \cite[Exercise 2.14]{MR2050138}, one can deduce that \eqref{eq:StokesWeak} is also well--posed for $p \in (P',2)$.
\end{remark}

Notice that Proposition~\ref{prop:OSwellposed} assumes only that the domain is Lipschitz. Finer results can be obtained provided more information on the domain is available. Since we are working on convex polytopes we have the following result; see \cite[Corollary 1.8]{MR2987056}.

\begin{proposition}[$L^p$--regularity]
\label{prop:MitreaWright}
Let $d \in \{2,3\}$ and $\Omega \subset \R^d$ be a convex polytope. If $p \in (1,2]$ and $\bef \in \bL^p(\Omega)$, then the solution of \eqref{eq:StokesStrong} is such that
\[
  \ue \in \bW^{2,p}(\Omega) \cap \bW^{1,p}_0(\Omega), \qquad \pe \in W^{1,p}(\Omega)/\R,
\]
with a corresponding estimate.
\end{proposition}

\section{Finite element approximation}
\label{sec:finite_element_approx}

We now introduce the discrete setting in which we will operate. We first introduce some terminology and a few basic ingredients and assumptions that will be common to all our methods.

\subsection{Triangulation and finite element spaces}
\label{sec:fem}
We denote by $\T = \{T\}$ a conforming partition, or mesh, of $\bar\Omega$ into closed simplices $T$ with size $h_T = \diam(T)$ and define $h = \max_{T \in \T} h_T$.  We assume that $\Tr = \{ \T \}_{h>0}$ is a collection of conforming and quasiuniform meshes \cite{CiarletBook,MR2050138}. For $T \in \T$, we define the \emph{star} or \emph{patch} associated with the element $T$ as
\begin{equation}
\label{eq:STstar}
  \mathcal{S} _T := \bigcup \{T' \in \T : T \cap T' \neq \emptyset \}.
\end{equation}

In the literature, several finite element approximations have been proposed and analyzed to approximate the solution to the Stokes problem \eqref{eq:StokesWeak} when the forcing term of the momentum equation is not singular; see, for instance, \cite[Section 4]{MR2050138}, \cite[Chapter II]{MR851383}, and references therein. Initially we shall not be specific about the type of finite element approximation that we are using. We will only state a set of assumptions that our discrete spaces need to satisfy. Given a mesh $\T \in \Tr$, we denote by $\V_h$ 
and $\P_h$ the finite element spaces that approximate the velocity field and the pressure, respectively, constructed over $\T$. We assume that, for every $p \in (1,\infty)$ and $\omega \in A_p$,
\[
  \V_h \subset \bW^{1,\infty}_0(\Omega) \subset \bW^{1,p}_0(\omega,\Omega), \qquad \P_h\subset L^\infty(\Omega)/\R \subset L^p(\omega,\Omega)/\R.
\]
In addition, we require that functions in $\V_h$ and $\P_h$ are locally polynomials of degree at least one and zero, respectively. Moreover, we need to assume that these spaces are compatible, in the sense that they satisfy weighted versions of the classical LBB condition \cite[Proposition 4.13]{MR2050138}. Namely, we assume that if $\omega \in A_p$ then, there exists a positive constant $\beta = \beta([\omega]_{A_p})$ such that, for all $\T \in \Tr$,
\begin{equation}
\label{eq:infsup_discrete}
  \begin{dcases}
    \inf_{ q_h \in \P_h} \sup_{ \bv_h \in \V_h } \frac{b(\bv_h, q_h)}{ \| \nabla \bv_h \|_{\bL^{p'}(\omega',\Omega)}  \| q_h \|_{L^p(\omega,\Omega)}  } \geq \beta, \\
    \inf_{ q_h \in \P_h} \sup_{ \bv_h \in \V_h } \frac{b(\bv_h, q_h)}{ \| \nabla \bv_h \|_{\bL^p(\omega,\Omega)}  \| q_h \|_{L^{p'}(\omega',\Omega)}  } \geq \beta.
  \end{dcases}
\end{equation}

\subsection{A quasi--interpolation operator}
\label{subsec:quasi}

Since our interest is to approximate rough functions the classical Lagrange interpolation operator cannot be applied. Instead, we can consider a variant of the quasi--interpolation operator analyzed in \cite{NOS3}. Its construction is inspired in the ideas developed by Cl\'ement \cite{MR0400739}, Scott and Zhang \cite{MR1011446}, and Dur\'an and Lombardi \cite{MR2164092}: it is built on local averages over stars and is thus well--defined for locally integrable functions; it also exhibits optimal approximation properties.

For $\T \in \Tr$, we let $X_h$ be the space of piecewise linear, continuous, functions over the mesh $\T$. For $w \in L^1(\Omega)$, we define $\Pi_{X_h} w \in X_h$ to be the interpolation operator of \cite{NOS3} onto piecewise linears. Define $\bX_h = [X_h \cap H^1_0(\Omega)]^d$.
For $\bv \in \bW^{1,1}_0(\Omega)$, we set $\Pi_{\V_h} \bv \in \bX_h \subset \V_h$ to be the operator $\Pi_{X_h}$ applied component--wise and accordingly modified to preserve boundary conditions.

To define an interpolant onto the pressure space $\P_h$ we distinguish two cases. If $\P_h$ contains piecewise constants, then, for $q \in L^1(\Omega)/\R$, we simply define $\Pi_{\P_h} q \in \P_h$ to be the local average of $q$. On the other hand, if $\P_h$ contains piecewise linears $\Pi_{\P_h}q = \Pi_{X_h} q + c_q$, where $c_q \in \R$ is chosen so that $\Pi_{\P_h}q \in \P_h$.
%
%

To alleviate notation, if there is no source of confusion, we shall use $\Pi_h$ to denote indistinctely $\Pi_{\V_h}$ or $\Pi_{\P_h}$. The properties of $\Pi_h$ are summarized below. For a proof we refer the reader to \cite[Section 5]{NOS3}. 

\begin{proposition}[stability and interpolation estimates]
Let $p \in (1,\infty)$, $\omega \in A_p$, and $T \in \T$. Then, for every $v \in W^{1,p}(\omega,\mathcal{S}_T)$, we have the local stability bound
\begin{equation}
\label{eq:local_stability}
 \| \GRAD \Pi_h v\|_{\bL^p(\omega,T)} \lesssim \| \GRAD v \|_{\bL^p(\omega,\mathcal{S}_T)}
\end{equation}
and the interpolation error estimate
\begin{equation}
\label{eq:interpolation_estimate_1}
\|  v - \Pi_h v \|_{L^p(\omega,T)} \lesssim h_T \| \GRAD v \|_{\bL^p(\omega,\mathcal{S}_T)}.
\end{equation}
The hidden constants, in \eqref{eq:local_stability} and \eqref{eq:interpolation_estimate_1}, are independent of $v$, $T$, and $h$.
\label{pro:interpolation}
\end{proposition}

This operator also enjoys the following approximation property \cite[Section 6]{NOS3}.

\begin{proposition}[interpolation in different metrics]
\label{prop:intH1wtoL2}
Assume that $\omega \in A_p$ is such that Proposition~\ref{prop:embedding} holds. Then, for every $v_p \in W^{1,p}(\omega,\mathcal{S}_T)$, we have that
\[
  \| v_p - \Pi_h v_p \|_{L^p(T)} \lesssim h_T^{1+d/p} \omega(\mathcal{S}_T)^{-1/p} \| \nabla v_p \|_{\bL^p(\omega,\mathcal{S}_T)}.
\]
Similarly, for $v_{p'} \in W^{1,p'}(\mathcal{S}_T)$, we have
\[
  \| v_{p'} - \Pi_h v_{p'} \|_{L^{p'}(\omega', T)} \lesssim h_T^{1-d/p'} \omega'(\mathcal{S}_T)^{1/p'} \| \nabla v_{p'} \|_{\bL^{p'}(\mathcal{S}_T)}.
\]
The hidden constants in the previous inequalities are independent of the functions being iterpolated, the cell $T$, and $h$.
\end{proposition}

\subsection{Approximate Green's function}
Let $z \in \Omega$ be such that $z \in \mathring{T}_{z}$ for some $T_{z} \in \T$. Let $\tilde \delta_z$ be a regularized Dirac delta satisfying the following properties:
\begin{enumerate}[1.]
  \item $\tilde \delta_{z} \in C_0^{\infty}(T_{z})$;
  \item $\int_{\Omega} \tilde \delta_{z} \diff x = 1$;
  \item $ \| \tilde \delta_{z}  \|_{L^{\infty}(T_{z})} \lesssim h_{T_z}^{-d}$;
  \item $\int_{\Omega}  \tilde \delta_{z} \bv_h \diff x = \bv_h(z)$ for all $\bv_h \in \V_h$. 
\end{enumerate}

We refer to \cite{MR0337032} and \cite[Exercise 8.1]{MR2373954} for a construction of such a function. Notice that, if $\bv_h = (v_h^1,\dots,v_h^d)^\intercal \in \V_h$ and $j \in \{1,\dots,d\}$, we have
\[
 \partial_{x_i} \bv_h^{j} (z) = \int_{\Omega} \partial_{x_i} \bv_h^j  \tilde \delta_{z} \diff x = -  \int_{\Omega} \bv_h^j   \partial_{x_i} \tilde \delta_{z} \diff x, \quad  i \in \{1,\dots,d\}.
\]

With these ingredients at hand, we define a regularized Green's function $(\bG,Q)$ as the solution to the following problem: Find $(\bG,Q) \in \bH_0^1(\Omega) \times L^2(\Omega)/\R$ such that
 \begin{equation}
\label{eq:StokesWeakGreen}
  \begin{dcases}
  a(\bG,\bv) + b(\bv,Q) =  \int_{\Omega} \tilde \delta_{z} \partial_{x_i} \bv^j \diff x& \forall \bv \in \bH^1_0(\Omega), \\
  b(\bG,q) = 0 &\forall q \in L^2(\Omega)/\R,
  \end{dcases}
\end{equation}
where $i, j \in \{ 1,\cdots,d\}$. Notice that the functions $\bG$ and $Q$ depend on $z$ and the indices $i$ and $j$. However, to alleviate notation we will omit this dependence.

We also define $(\bG_h,Q_h)$, the Stokes projection of $(\bG,Q)$, as the solution to the discrete problem: Find $(\bG_h,Q_h) \in \V_h \times \P_h$ such that
\begin{equation}
\label{eq:StokesDiscreteGreen}
  \begin{dcases}
    a(\bG_h,\bv_h) + b(\bv_h,Q_h) =  \int_{\Omega} \tilde \delta_{z} \partial_{x_i} \bv_h^j \diff x   & \forall \bv_h \in \V_h, \\
    b(\bG_h,q_h) = 0 &\forall q_h \in \P_h.
  \end{dcases}
\end{equation}

Let $R$ be a fixed positive number such that for any $x \in \bar \Omega$ the ball $B(x,R)$ contains $\Omega$. For $y \in \Omega$, we define the weight function $\sigma_{y}$, introduced by Natterer \cite{MR0474884}, as
\begin{equation}
 \label{eq:sigma}
 \sigma_{y} (x) = \left( | x - y|^2 + \left( \kappa h \right)^2 \right)^{1/2},
\end{equation}
where $\kappa > 1$ is a parameter independent of $h$ but such that $\kappa h \leq R$; see \cite[Section 1.7]{MR3422453}. We recall that this weight verifies \cite[inequality (0.18)]{MR2121575}
\begin{equation}
\label{eq:integrateNattererweight}
  \int_\Omega \sigma_y^{-d-\lambda} \diff x \lesssim h^{-\lambda},
  \quad \lambda \in (0,1).
\end{equation}

We shall assume that if $\nu \in (0,1/2)$, $ 0 < \lambda < \nu/2$, $\mu = d + \lambda$, and $\T$ is quasiuniform, then there exists $\kappa_1 > 1$ such that for all $\kappa \geq \kappa_1$ and for all meshsizes $h > 0$ such that $\kappa h \leq R$, we have 
\begin{equation}
\label{eq:bG-bGT}
 \sup_{y \in \Omega} \left \| \sigma_{y}^{\frac{\mu}{2}} \nabla (\bG - \bG_h) \right \|_{\bL^2(\Omega)} \lesssim h^{\lambda/2}.
\end{equation}
Examples of spaces that satisfy this assumption will be presented below.

\section{Discrete stability estimates in weighted spaces}
\label{sec:discrete_stability}
Let $p \in (1,\infty)$, $\omega \in A_p$, $(\ue,\pe) \in \bW^{1,p}_0(\omega,\Omega)\times L^p(\omega,\Omega)/\R$ with $\ue$ solenoidal, and the pair $(\ue_h,p_h) \in \V_h \times\P_h$ be the finite element approximation of $(\ue,\pe)$. Our goal in this section is to, on the basis of the weighted compatibility conditions \eqref{eq:infsup_discrete}, derive the weighted stability estimate \eqref{eq:WeightedStability}. 
To do so, we must place some restrictions on the range of the integrability $p$ and the weight $\omega$. We codify these in the following assumption
\begin{equation}
\label{eq:condtionpandw}
\tag{S}
\begin{dcases}
  p \in (2,\infty) & \implies \omega \in A_1, \\
  p  = 2 & \implies \omega \in A_1, \text{ or } \omega^{-1} \in A_2(\Omega) \cap A_1, \\
  p \in (P', 2] & \implies \omega' \in A_{p'}(\Omega) \cap A_1,
\end{dcases}
\end{equation}
where $P$ is as in Proposition \ref{prop:OSwellposed} and $P'$ is its H\"older conjugate.

\begin{theorem}[weighted stability estimate]
\label{thm:weightstab}
Let $d \in \{2,3\}$ and $\Omega \subset \R^d$ be an open convex polytope. Assume that \eqref{eq:condtionpandw} holds and that $(\ue,\pe) \in \bW^{1,p}_0(\omega,\Omega) \times L^p(\omega,\Omega)/\R$ with $\ue$ solenoidal. Let $(\ue_h,\pe_h) \in \V_h \times \P_h$ be its finite element Stokes projection. If the spaces $(\V_h,\P_h)$ satisfy \eqref{eq:infsup_discrete} and \eqref{eq:bG-bGT}, then estimate \eqref{eq:WeightedStability} holds. The hidden constant in this estimate is independent of $(\ue,\pe)$, $(\ue_h,\pe_h)$, and $h$.
\end{theorem}
\begin{proof} We begin by noticing that, by density, it suffices to show the estimate assuming that $\ue$ and $\pe$ are smooth.

We split the proof in several steps.

\begin{enumerate}[1.]
  \item Assume that we have already shown that
  \begin{equation}
  \label{eq:WeightedBu}
    \| \nabla \ue_h \|_{\bL^p(\omega,\Omega)}  \lesssim  \| \nabla \ue \|_{\bL^p(\omega,\Omega)}  
      +  \| \pe \|_{L^p(\omega,\Omega)}.
  \end{equation}
Utilizing the first discrete inf--sup condition of \eqref{eq:infsup_discrete} and that $(\ue_h,\pe_h)$ solves \eqref{eq:Stokesproj}, we arrive at
  \[
   \| \pe_h \|_{L^p(\omega,\Omega)} \lesssim 
   \sup_{ \bv_h \in \V_h } \frac{b (\bv_h, \pe_h) }{ \| \nabla \bv_h \|_{\bL^{p'}(\omega',\Omega)}}  =
    \sup_{ \bv_h \in \V_h } \frac{ a(\ue,\bv_h) - a(\ue_h, \bv_h) + b(\bv_h, \pe)}{ \| \nabla \bv_h \|_{\bL^{p'}(\omega',\Omega)}},
   \]
  which immediately yields
  \[
    \| \pe_h \|_{L^p(\omega,\Omega)}  \lesssim \| \nabla \ue \|_{\bL^p(\omega,\Omega)} +  \| \pe \|_{L^p(\omega,\Omega)} + \| \nabla \ue_h \|_{\bL^p(\omega,\Omega)}.
  \]
  This, in view of \eqref{eq:WeightedBu}, implies the desired bound for $\| \pe_h \|_{L^p(\omega,\Omega)}$.
  
  \item Assume that $p\geq 2$ and $\omega \in A_1$. Set $\bv_h = \ue_h$ in \eqref{eq:StokesDiscreteGreen} to arrive at
  \[
    a(\bG_h, \ue_h)
    = \int_{\Omega} \tilde \delta_{z} \partial_{x_i} \ue_h^j \diff x =  \partial_{x_i} \ue_h^j (z).
  \]
  Set now $\bv_h = \bG_h$ in \eqref{eq:Stokesproj} and use that $b(\bG_h,q_h) = 0$ for all $q_h \in \P_h$ to obtain
  \begin{equation}
  \label{eq:AuxEstimateISA}
   a(\ue_h,\bG_h) =  a(\ue, \bG_h) + b(\bG_h, \pe) .
  \end{equation}
  Using that $b(\bG,\pe) = 0$, we can thus conclude the identity
  \[
    a(\ue_h,\bG_h) = a(\ue,\bG_h) + b(\bG_h,\pe) =a(\ue,\bG_h - \bG) + b(\bG_h - \bG,\pe) + a(\ue,\bG).
  \]
  Since the bilinear form $a$ is symmetric, we have
  \begin{align*}
  \partial_{x_i} \ue_h^j (z) & = a(\ue,\bG_h - \bG) + b(\bG_h - \bG,\pe) + a(\ue,\bG)
  \\
  & = a(\ue,\bG_h - \bG) + b(\bG_h - \bG,\pe) + \int_{\Omega} \tilde \delta_{z} \partial_{x_i} \ue^j  \diff x.
  \end{align*}
  Notice that here we used the smoothness assumption on $\ue$ to be able to assert that this is an admissible test function in \eqref{eq:StokesWeakGreen}.

  Let now $\bE = \bG - \bG_h$. The previous equality implies that
  \begin{multline*}
   \int_{\Omega} \omega |\partial_{x_i} \ue_h^j|^p \diff z \lesssim 
   \int_{\Omega} \omega\left[\int_{\Omega}  \nabla \ue : \nabla  \bE \diff x \right]^p  \diff z \\
   + 
   \int_{\Omega} \omega \left[\int_{\Omega}  \pe \DIV \bE \diff x \right]^p \diff z 
   +
   \int_{\Omega} \omega  \left[\fint_{T_{z}}  |\nabla \ue| \diff x \right]^p \diff z =: \mathrm{I} + \mathrm{II} + \mathrm{III},
  \end{multline*}
  where we have used that $\tilde \delta_{z}$ is supported on $T_{z}$ and that $\| \tilde \delta_{z}\|_{L^{\infty}(\Omega)} \lesssim h^{-d}$.
    
  We estimate the terms $\mathrm{I}$, $\mathrm{II}$, and $\mathrm{III}$ with the help of \eqref{eq:bG-bGT}, similar arguments to those developed in the proof of \cite[Theorem 3.1]{DDO:17}, and modifications inspired by \cite{MR645661}. We begin by controlling the term $\mathrm{III}$. Since the weight $\omega \in A_1 \subset A_p$, we utilize that the Hardy--Littlewood maximal operator $\mathcal{M}$ is continuous from $L^p(\omega,\R^d)$ to $L^p(\omega,\R^d)$ to arrive at
  \[
    \mathrm{III} =  \int_{\Omega} \omega  \left[ \fint_{T_{z}}  |\nabla \ue |\diff x \right]^p \diff z \lesssim \int_{\Omega} \omega \mathcal{M}(|\nabla \ue|)^p \diff z \lesssim \int_{\Omega} \omega  |\nabla \ue|^p \diff z.
  \]

  We now control $\mathrm{I}$ and $\mathrm{II}$. Using the weight $\sigma_{z}$, defined in \eqref{eq:sigma}, and its property \eqref{eq:integrateNattererweight} we have that for any $\lambda \in (0,1)$
  \[
   \int_{\Omega}  \nabla \ue :\nabla \bE \diff x \lesssim h^{-\lambda(p-2)/(2p)} \left(\int_{\Omega} \sigma_{z}^{-d-\lambda} |\nabla \ue|^p \diff x \right)^{1/p} \left(\int_{\Omega} \sigma_{z}^{d+\lambda}|\nabla \bE|^2 \diff x \right)^{1/2}
  \]
  and 
  \[
   \int_{\Omega}  \pe \DIV \bE \diff x \lesssim h^{-\lambda(p-2)/(2p)} \left(\int_{\Omega} \sigma_{z}^{-d-\lambda} | \pe|^p \diff x \right)^{1/p} \left(\int_{\Omega} \sigma_{z}^{d+\lambda}|\DIV \bE|^2 \diff x \right)^{1/2}.
  \]
  Thus, we have that
  \[
  \mathrm{I} + \mathrm{II} \lesssim h^{-\lambda(p-2)/2}\int_{\Omega} \omega \left(\int_{\Omega} \sigma_{z}^{d+\lambda}|\nabla \bE|^2 \diff x \right)^{p/2} \left(\int_{\Omega} \frac{ |\nabla \ue|^p +| \pe|^p }{\sigma_{z}^{d+\lambda}} \diff x \right)\diff z.
  \]
  Assume now that $0 < \lambda < \nu/2$ with $\nu \in (0,1/2)$. In this case estimate \eqref{eq:bG-bGT} immediately yields
  \[
    h^{-\lambda(p-2)/2}\left( \int_{\Omega} \sigma_{z}^{d+\lambda}|\nabla \bE|^2 \diff x \right)^{p/2} \lesssim h^{\lambda}.
  \]
  In addition, the arguments developed in the proof of \cite[Theorem 3.1]{DDO:17} yield
  \[
  \int_{\Omega} \frac{ h^{\lambda} \omega(z)}{(|x-z|^2 + (\kappa h)^2 )^{(d+\lambda)/2}}\diff z  \lesssim \mathcal{M}\omega(x) \lesssim \omega(x),
  \] 
  where, in the last step, we used \eqref{eq:A1usingMax}. In conclusion, we obtained that
  \begin{align*}
  \mathrm{I} + \mathrm{II} &\lesssim \int_{\Omega}\int_{\Omega}  \frac{ \omega(z) h^{\lambda} }{ (|x-z|^2 + (\kappa h)^2 )^{(d+\lambda)/2} } \diff z \left(|\nabla \ue(x)|^p +| \pe(x)|^p \right) \diff x  
  \\
  &\lesssim  \int_{\Omega} \omega(x) \left(|\nabla \ue(x)|^p +| \pe(x)|^p \right) \diff x.
  \end{align*}
  A collection of the estimates for the terms $\mathrm{I}$,  $\mathrm{II}$, and $\mathrm{III}$ yield \eqref{eq:WeightedBu} when $p \geq 2$.
  
  \item It remains to consider the case $p \in (P',2]$ with $\omega' \in A_{p'}(\Omega) \cap A_1$. Notice that $p' = p / (p-1) \geq 2$ so that, as in \cite[Corollary 3.3]{DDO:17}, we will reduce our considerations to the previous case. Since $p' \in [2,P)$ and $\omega' \in A_{p'}(\Omega)$ then, as Proposition~\ref{prop:OSwellposed} shows, for every $\bg \in \bW^{-1,p'}(\omega',\Omega)$ we conclude that the Stokes problem
  \[
    \begin{dcases}
      -\Delta \bvarphi_\bg + \GRAD \psi_\bg = \bg, & \text{in } \Omega, \\
      \DIV \bvarphi_\bg = 0, & \text{in } \Omega, \\
      \bvarphi_\bg = 0, & \text{on } \partial\Omega,
    \end{dcases}
  \]
  is well-posed in $\bW^{1,p'}_0(\omega',\Omega) \times L^{p'}(\omega',\Omega)/\R$. So that we have the estimate
  \[
    \| \GRAD \bvarphi_\bg \|_{\bL^{p'}(\omega',\Omega)} + \| \psi_\bg \|_{L^{p'}(\omega',\Omega)} \lesssim \| \bg \|_{\bW^{-1,p'}(\omega',\Omega)}.
  \]
  Let $(\bvarphi_{\bg,h} , \psi_{\bg,h}) \in \V_h \times \P_h$ be the Stokes projection of $(\bvarphi_\bg,\psi_\bg)$ we have
  \begin{align*}
    \| \GRAD \ue_h \|_{\bL^p(\omega,\Omega)} &= \sup_{\bg \in \bW^{-1,p'}(\omega',\Omega)} \frac{ \langle \bg, \ue_h \rangle }{\| \bg \|_{\bW^{-1,p'}(\omega',\Omega)}} \\ 
    &= \sup_{\bg \in \bW^{-1,p'}(\omega',\Omega)} \frac{ a(\ue_h,\bvarphi_\bg) + b(\ue_h, \psi_\bg) }{\| \bg \|_{\bW^{-1,p'}(\omega',\Omega)}} \\
    &= \sup_{\bg \in \bW^{-1,p'}(\omega',\Omega)} \frac{ a(\bvarphi_{h,\bg}, \ue_h) + b(\ue_h, \psi_{h,\bg}) }{\| \bg \|_{\bW^{-1,p'}(\omega',\Omega)}} \\
    &= \sup_{\bg \in \bW^{-1,p'}(\omega',\Omega)} \frac{ a(\ue, \bvarphi_{h,\bg}) + b(\ue, \psi_{h,\bg}) }{\| \bg \|_{\bW^{-1,p'}(\omega',\Omega)}}.
  \end{align*}
  The stability of the Stokes projection in $\bW^{1,p'}(\omega',\Omega) \times L^{p'}(\omega',\Omega)$ and the bound on $(\bvarphi_\bg,\psi_\bg)$ yield
  \[
    \|  \nabla \ue_h \|_{\bL^p(\omega,\Omega)} \lesssim \| \nabla \ue\|_{\bL^{p}(\omega,\Omega)}.
  \]
\end{enumerate}
The proof is thus complete.
\end{proof}


As usual, the a priori estimate \eqref{eq:WeightedStability} implies a best approximation result \emph{\`a la} Cea.

\begin{corollary}[best approximation]
\label{cor:Cea}
In the setting of Theorem~\ref{thm:weightstab}, assume, in addition, that $p \in (P',P)$, and $\omega \in A_p(\Omega)$. Then we have that
\begin{multline*}
  \| \nabla (\ue- \ue_h) \|_{\bL^p(\omega,\Omega)}  +  \| \pe - \pe_h \|_{L^p(\omega,\Omega)} \lesssim \inf_{\bw_h \in \V_h} \| \nabla (\ue- \bw_h) \|_{\bL^p(\omega,\Omega)}  \\
  +  \inf_{r_h \in \P_h} \| \pe - r_h \|_{L^p(\omega,\Omega)},
\end{multline*}
where the hidden constant is independent of $(\ue,\pe)$, $(\ue_h,p_h)$, and $h$.
\end{corollary}
\begin{proof}
The proof is rather standard but we reproduce it here for the sake of completeness. Notice that, if $\bw_h \in \V_h$ and $r_h \in \P_h$ are arbitrary, by linearity of \eqref{eq:Stokesproj} we obtain that, for all $(\bv_h,q_h) \in \V_h \times \P_h$ we have
\[
  \begin{dcases}
    a(\ue_h - \bw_h,\bv_h) + b(\bv_h,\pe_h-r_h) = a(\ue - \bw_h,\bv_h) + b(\bv_h,\pe-r_h), \\
    b(\ue_h-\bw_h,q_h) = b(\ue - \bw_h,q_h).
  \end{dcases}
\]

Let now $(\bvarphi,\psi) \in \bW^{1,p}_0(\omega,\Omega) \times L^p(\omega,\Omega)/\R$ be the unique solution of
\[
  \begin{dcases}
    a(\bvarphi,\bv) + b(\bv,\psi) = a(\ue - \bw_h,\bv) + b(\bv,\pe-r_h), & \forall \bv \in \bW^{1,p'}_0(\omega',\Omega), \\
    b(\bvarphi,q) = b(\ue - \bw_h,q), & \forall q \in L^{p'}(\omega',\Omega)/\R.
  \end{dcases}
\]
As shown in Proposition~\ref{prop:OSwellposed}, the assumptions on the integrability index and the weight allow us to conclude that this problem is well posed and we have the estimate
\begin{equation}
\label{eq:pdetrick}
  \| \nabla \bvarphi \|_{\bL^p(\omega,\Omega)}  +  \| \psi \|_{L^p(\omega,\Omega)} \lesssim \| \nabla (\ue - \bw_h) \|_{\bL^p(\omega,\Omega)}  +  \| \pe - r_h \|_{L^p(\omega,\Omega)}.
\end{equation}

Notice now that $(\ue_h-\bw_h, \pe_h - r_h) \in \V_h \times \P_h$ is nothing but the finite element approximation of $(\bvarphi ,\psi ) \in \bW^{1,p}_0(\omega,\Omega) \times L^p(\omega,\Omega)/\R$. This, in conjunction with Theorem~\ref{thm:weightstab} and \eqref{eq:pdetrick} then yields
\begin{multline*}
  \| \nabla (\ue_h - \bw_h) \|_{\bL^p(\omega,\Omega)}  +  \| \pe_h - r_h \|_{L^p(\omega,\Omega)} \lesssim \| \nabla (\ue - \bw_h) \|_{\bL^p(\omega,\Omega)}  \\ +  \| \pe - r_h \|_{L^p(\omega,\Omega)}.
\end{multline*}

Conclude with the triangle inequality.
\end{proof}

\section{Error estimates}
\label{sec:error_estimates}

We now provide a $\bL^p(\Omega)$--error estimate for the error approximation of the velocity field. For that, obviously, one needs to assume that Proposition~\ref{prop:embedding} holds, so that $\ue \in \bL^p(\Omega)$.

In what follows, for a weight $\omega$,  we denote by $\omega(h) = \sup_{T \in \T} \omega(T)$. The main error estimate is provided below. 

\begin{theorem}[error estimate]
\label{thm:errest}
Let $p \in [2,P)$
and $\omega \in A_p(\Omega)$ be such that condition \eqref{eq:condtionpandw} holds. Assume, in addition, that the compatibility condition required for Proposition~\ref{prop:embedding} to be valid holds. Let $(\ue,\pe) \in \bW^{1,p}_0(\omega,\Omega) \times L^p(\omega,\Omega)/\R$ with $\ue$ solenoidal, and let $(\ue_h,\pe_h) \in \V_h \times \P_h$ be its Stokes projection, defined as the solution of \eqref{eq:Stokesproj}. In this setting, we have that
\begin{equation}
 \label{eq:L2ErrorEstimate}
 \| \ue - \ue_h\|_{\bL^p(\Omega)} \lesssim h^{1+d/p} \omega(h)^{-1/p}\left( \| \nabla \ue  \|_{\bL^p(\omega,\Omega)} + \| \pe  \|_{L^p(\omega,\Omega)} \right),
\end{equation}
where the hidden constant is independent of $(\ue,\pe)$, $(\ue_h,\pe_h)$, and $h$.
\end{theorem}
\begin{proof}
We proceed in several steps on the basis of a duality argument. 
\begin{enumerate}[1.]
  \item We begin by recalling that, owing to Proposition~\ref{prop:MitreaWright}, for every $ t\in (1,2]$ we have that, if $\bg \in \bL^t(\Omega)$, the Stokes problem: find $(\bvarphi_\bg,\psi_\bg) \in \bW^{1,t}_0(\Omega) \times L^t(\Omega)/\R$
  \begin{equation}
  \label{eq:StokesDualWeak}
    \begin{dcases}
    a(\bvarphi_\bg,\bv) + b(\bv,\psi_\bg) =  \int_\Omega \bg \cdot \bv \diff x & \forall \bv \in \bW^{1,t'}_0(\Omega), \\
    b(\bvarphi_\bg,q) = 0 &\forall q \in L^{t'}(\Omega)/\R,
    \end{dcases}
  \end{equation}
  is well--posed, $(\bvarphi_\bg,\psi_\bg) \in \bW^{2,t}(\Omega) \times W^{1,t}(\Omega)$, and 
  \begin{equation}
  \label{eq:dualitysmooth}
    \| \bvarphi_\bg \|_{\bW^{2,t}(\Omega)} + \| \psi_\bg \|_{W^{1,t}(\Omega)} \lesssim \| \bg \|_{\bL^t(\Omega)}.
  \end{equation}

  \item Since $p \geq2$ and $\omega \in A_p(\Omega)$ satisfies the compatibility condition of Proposition~\ref{prop:embedding} we can use the results of the previous step with $t = p'$ and the embedding results of Proposition \ref{prop:embedding} to conclude that 
  \[
   (\bvarphi_\bg, \psi_\bg) \in \bW^{2,p'}(\Omega) \cap \bW_0^{1,p'}(\omega',\Omega) \times W^{1,p'}(\Omega) \cap L^{p'}(\omega',\Omega)
  \]
  with an estimate.

  \item Let $\bg = |\ue - \ue_h|^{p-2}(\ue - \ue_h)$ and note that $\| \bg \|_{L^{p'}(\Omega)} = \| \ue - \ue_h \|_{\bL^p(\Omega)}^{p-1}$, which is finite given the assumption on $\omega$ and the embedding results of Proposition \ref{prop:embedding}.

  \item With this choice of $\bg$ fixed, we would like to set $\bv = \ue - \ue_h$ in \eqref{eq:StokesDualWeak} to obtain
  \begin{equation}
  \label{eq:DualityFirst}
     \| \ue - \ue_h \|_{\bL^p(\Omega)}^p = a( \ue - \ue_h, \bvarphi_\bg ) + b( \ue - \ue_h ,\psi_\bg).
  \end{equation}
  However, since $p \geq 2$, $\ue- \ue_h \notin \bW^{1,p}_0(\Omega)$ so that \eqref{eq:DualityFirst} must be justified by a density argument. Namely, let $\mathbf{w}_n \in \mathbf{C}_0^{\infty}(\Omega)$ be such that $\mathbf{w}_n \to \ue - \ue_h$ in  $\bW^{1,p}_0(\omega,\Omega)$. Since $\mathbf{w}_n \in \mathbf{C}_0^{\infty}(\Omega) \subset \bW^{1,p}_0(\Omega)$, we set $\bv = \mathbf{w}_n$ in \eqref{eq:StokesDualWeak} and arrive at
  \begin{equation}
  \label{eq:Dualitywn}
     a(\mathbf{w}_n,\boldsymbol{\varphi}_\bg) + b(\mathbf{w}_n,\psi_\bg) =  \int_\Omega |\ue - \ue_h|^{p-2}(\ue - \ue_h) \cdot \mathbf{w}_n \diff x.
  \end{equation}
  Now, since $\bvarphi_\bg \in \bW^{1,p'}_0(\omega',\Omega)$, 
  \[
   | a(\ue - \ue_h,\bvarphi_\bg) - a(\mathbf{w}_n,\bvarphi_\bg) | \leq \| \nabla \bvarphi_\bg \|_{\bL^{p'}(\omega',\Omega)}
    \| \nabla(\ue - \ue_h - \mathbf{w}_n) \|_{\bL^p(\omega,\Omega)} \to 0
  \]
  as $n \uparrow \infty$. Similar arguments reveal that $| b(\ue - \ue_h,\psi_\bg) - b(\mathbf{w}_n,\psi_\bg)| \to 0$ as $n \uparrow \infty$. Finally, in view of the continuous embedding $\bW^{1,p}_0(\omega,\Omega) \hookrightarrow \bL^p(\Omega)$, the right hand side of \eqref{eq:Dualitywn} converges to $ \|\ue - \ue_h \|_{\bL^p(\Omega)}^p$. These arguments yield \eqref{eq:DualityFirst}.
  
  \item From \eqref{eq:DualityFirst} and \eqref{eq:Stokesproj} we have, for an arbitrary pair $(\bw_h,r_h) \in \V_h \times \P_h$,
  \[
    \| \ue - \ue_h \|_{\bL^p(\Omega)}^p = a(\ue - \ue_h, \bvarphi_\bg - \bw_h  ) - b( \ue_h ,\psi_\bg - r_h) - b(\bw_h, \pi - \pi_h),
  \]
  where we also used that $\ue$ is solenoidal. Set now $\bw_h = \bvarphi_{\bg,h}$ and $r_h = \psi_{\bg,h}$, \ie the Stokes projection of $(\bvarphi_\bg,\psi_\bg)$. Galerkin orthogonality once again yields
  \[
    \| \ue - \ue_h \|^p_{\bL^p(\Omega)} = a(\ue , \boldsymbol{\varphi}_{\bg} - \boldsymbol{\varphi}_{\bg,h} ) + b(\boldsymbol{\varphi}_{\bg}-\boldsymbol{\varphi}_{\bg,h},\pe ). 
  \]
  Consequently
  \begin{equation*}
  \| \ue - \ue_h \|^p_{\bL^p(\Omega)} \lesssim 
  \| \nabla( \bvarphi_{\bg} - \bvarphi_{\bg,h} ) \|_{\bL^{p'}(\omega',\Omega)} 
   \left( \| \nabla \ue  \|_{\bL^p(\omega,\Omega)} + \| \pe  \|_{L^p(\omega,\Omega)} \right).
  \end{equation*}
  
  \item As a final step we must bound the first term on the right hand side of the previous estimate. Notice that, with $t = p'<2$, and $\varrho: = \omega'$ what we are trying to estimate is the error in the velocity component of the Stokes projection in $\bW^{1,t}_0(\varrho,\Omega) $. This means that, since $t<2$, we can apply Corollary~\ref{cor:Cea} provided condition \eqref{eq:condtionpandw} holds, that is 
  \[
   \rho' \in A_{t'}(\Omega) \Leftrightarrow (\omega')^{-t'/t} \in A_{p''}(\Omega) \Leftrightarrow (\omega^{-p'/p})^{-p/p'} \in A_p(\Omega) \Leftrightarrow \omega \in A_p(\Omega),
  \]
  and
  \[
    \varrho' \in A_1 \Leftrightarrow \left( \omega' \right)^{-t'/t} \in A_1 \Leftrightarrow \left( \omega^{-p'/p} \right)^{-p/p'} \in A_1 \Leftrightarrow \omega \in A_1,
  \]
  which is true by assumption. The best approximation result of Corollary~\ref{cor:Cea}, the interpolation estimates of Proposition~\ref{prop:intH1wtoL2}, and the regularity estimate given in \eqref{eq:dualitysmooth} then yield
  \[
    \| \nabla( \bvarphi_{\bg} - \bvarphi_{\bg,h} ) \|_{\bL^{p'}(\omega',\Omega)} 
    \lesssim h^{1-d/{p'}} \omega'(h)^{1/p'} \| \ue - \ue_h \|_{\bL^p(\Omega)}^{p-1}.
  \]
  Conclude by observing that, since $\omega \in A_p$, for we have that
  \[
    \omega'(T)^{1/p'} \lesssim h^d \omega(T)^{-1/p}, \quad \forall T \in \T.
  \]

\end{enumerate}
This concludes the proof
\end{proof}

\subsection{Application: The Stokes problem with delta sources}
\label{subsec:application}

Let us now, as an application, show how Theorem~\ref{thm:errest} can be applied to the case of singular forces described in item \ref{item1} of Section~\ref{sub:weights}. Assume that $\calZ \subset \Omega$ with $\# \calZ < \infty$, \ie it is a finite collection of points. We now define
\begin{equation}
\label{eq:fZ}
  \bef_\calZ = \sum_{z \in \calZ} \bF_z \delta_z,
\end{equation}
with $\bF_z \in \R^d$. We begin by establishing the suitable functional framework.

\begin{proposition}[$\bef_\calZ \in \bH^{-1}(\dist{\calZ}^{\alpha},\Omega)$]
\label{prop:dual}
Assume that $\alpha \in (d-2,d)$, then $\dist{\calZ}^{\alpha} \in A_2(\Omega)$, $\dist{\calZ}^{-\alpha} \in A_2(\Omega) \cap A_1$, and $\bef_\calZ \in \bH^{-1}(\dist{\calZ}^{\alpha},\Omega)$.
\end{proposition}
\begin{proof}
The bounds on $\alpha$ guarantee that $\dist{\calZ}^{\alpha} \in A_2(\Omega)$ and $\dist{\calZ}^{-\alpha} \in A_2(\Omega)$. In addition, since $d-2\geq 0$, we have that $\dist{\calZ}^{-\alpha} \in A_1$.

Now, owing to \cite[Remark 21.19]{MR2305115}, a compactly supported Radon measure $\nu$ belongs to the dual of $H^1_0(\omega,\Omega)$ if
\[
  \int_\Omega \int_0^r \frac{t^2\nu(B(x,t))}{\omega(B(x,t))} \frac{\diff t}t \diff \nu(x) < \infty
\]
for some $r>0$. Setting $\nu = \sum_{z \in \calZ} \delta_z$ and $\omega = \dist{\calZ}^{-\alpha}$ we get
\[
  \int_\Omega \int_0^r \frac{t^2\nu(B(x,t))}{\omega(B(x,t))} \frac{\diff t}t \diff \nu(x) \lesssim \sum_{z \in \calZ} \int_0^r \frac{t}{t^{d-\alpha}} \diff t,
\]
which is finite provided $d-2<\alpha$. 
\end{proof}

The previous result shows that, if $\bef = \bef_\calZ$ in \eqref{eq:StokesStrong}, then this problem has a unique solution $(\ue,\pe) \in \bH^1_0(\dist{\calZ}^\alpha,\Omega) \times L^2(\dist{\calZ}^\alpha,\Omega)/\R$. The following result is the missing ingredient to obtain error estimates via Theorem~\ref{thm:errest}.

\begin{proposition}[embedding]
\label{prop:embeddingdelta}
If $\alpha \in (d-2,2)$, then $\bH^1_0(\dist{z}^{\alpha},\Omega) \hookrightarrow \bL^2(\Omega)$.
\end{proposition}
\begin{proof}
We only need to verify the condition of Proposition~\ref{prop:embedding}. In this case, we have
\[
  \frac{r^{2+d}}{R^{2+d}} \frac{\dist{z}^{\alpha}(B(x,R))}{\dist{z}^{\alpha}(B(x,r))} \approx \frac{r^{2+d}}{R^{2+d}}\frac{ R^{d+\alpha}}{r^{d+\alpha}} = \left(\frac{r}{R}\right)^{2-\alpha}.
\]
The provided bounds on $\alpha$ guarantee that this ratio is uniformly bounded.
\end{proof}

We can now obtain an error estimate. Notice that since $\dist{\calZ}^{-\alpha} \in A_2(\Omega) \cap A_1$, the results of Theorem \ref{thm:weightstab} and Corollary \ref{cor:Cea} apply.

\begin{corollary}[error estimate]
Let $\alpha \in (d-2,2)$ and $(\ue,\pe) \in \bH^1_0(\dist{\calZ}^\alpha,\Omega) \times L^2(\dist{\calZ}^\alpha,\Omega)/\R$ solve \eqref{eq:StokesStrong} with $\bef = \bef_\calZ$. Let $(\ue_h,\pe_h)$ be the finite element approximation of $(\ue,\pe)$. In the setting of Theorem~\ref{thm:errest} we have, for every $\vare >0$,
\[
  \| \ue - \ue_h\|_{\bL^2(\Omega)} \lesssim h^{2-d/2-\vare} \left( \| \nabla \ue  \|_{\bL^2(\dist{\calZ}^\alpha,\Omega)} + \| \pe  \|_{L^2(\dist{\calZ}^{\alpha},\Omega)} \right),
\]
where the hidden constant does not depend on $\ue$, $\pe$, nor $h$, but blows up as $\vare \downarrow 0$.
\end{corollary}
\begin{proof}
Proposition~\ref{prop:dual} guarantees that there is a unique pair $(\ue,\pe) \in \bH^1_0(\dist{\calZ}^\alpha,\Omega) \times L^2(\dist{\calZ}^\alpha,\Omega)/\R$ that solves \eqref{eq:StokesStrong}. In addition, Proposition~\ref{prop:embeddingdelta} guarantees that $\ue \in \bL^2(\Omega)$. The rest is just an application of Theorem~\ref{thm:errest}. In this case, we have that
\[
  h^{1+d/2}\omega(h)^{-1/2} = h^{1+d/2} h^{-d/2-\alpha/2} = h^{1-\alpha/2}
\]
and
\[
  \alpha \in (d-2,2) \implies 1-\frac\alpha2 \in \left(0,2-\frac{d}2\right).
\]
The blowup of the constants is due to the fact that in the limiting case the embedding $\bH^1_0(\dist{z}^{\alpha},\Omega) \hookrightarrow \bL^2(\Omega)$ does no longer hold.
\end{proof}

We conclude by commenting that via similar techniques we can consider the cases described in items \ref{item2} and \ref{item3} of Section~\ref{sub:weights}.

\section{Examples of suitable pairs}
\label{sec:Pairs}

To conclude our analyisis, we study some  pairs that satisfy assumptions \eqref{eq:infsup_discrete}, \eqref{eq:bG-bGT} so that the theory we have presented above applies.

We begin with a continuous weighted inf--sup condition that immediately follows from the existence of a right inverse of the divergence.

\begin{lemma}[continuous weighted inf--sup]
Let $p \in (1,\infty)$ and $\omega \in A_p$. For all $q \in L^{p'}(\omega',\Omega)/\R$ we have that
\begin{equation}
\label{eq:InfSupContinuous}
  \| q \|_{ L^{p'}(\omega',\Omega)} \lesssim \sup_{\bv \in \bW^{1,p}_0(\omega,\Omega)} \frac{b(\bv, q)}{\| \nabla \bv \|_{\bL^p(\omega,\Omega)}},
\end{equation}
where the hidden constant depends only on $\Omega$ and $[\omega]_{A_p}$, but not on $q$.
\end{lemma}
\begin{proof}
Let $q \in L^{p'}(\omega',\Omega)/\R$ and we define $\tilde r = \omega' |q|^{p'/p}\signum(q)$. Notice that
\[
 \| \tilde r\|^p_{L^{p}(\omega,\Omega)} = \int_{\Omega} \omega |\tilde r|^p \diff x =  \int_{\Omega} \omega^{1-p'} |q|^{p'} \diff x = \| q \|_{L^{p'}(\omega',\Omega)}^{p'},
\]
so that $\tilde r \in L^p(\omega,\Omega)$ and, since $\Omega$ is bounded $\tilde r \in L^1(\Omega)$. Consequently, we can set $r = \tilde r - \fint_\Omega \tilde r \diff x$ and we conclude that $r \in L^p(\omega,\Omega)/\R$ with
\[
  \| r \|_{L^p(\omega,\Omega)} \lesssim \| q \|_{L^{p'}(\omega',\Omega)}^{p'-1}.
\]
Our final initial observation is that, since $q$ has zero mean,
\[
  \int_\Omega q r \diff x = \int_\Omega q \tilde r \diff x = \int_\Omega |q|^{p'} \omega' \diff x= \| q \|_{L^{p'}(\omega',\Omega)}^{p'}.
\]

Recall now that there is $\bw \in \bW^{1,p}_0(\omega,\Omega)$ such that
\[ 
  \DIV \bw = r, \qquad \| \nabla \bw \|_{\bL^p(\omega,\Omega)} \lesssim \| r \|_{L^p(\omega,\Omega)},
\]
where the constant in the estimate is independent of $r$; see \cite[Theorem 3.1]{MR2731700}, \cite[Theorem 1]{MR2548872}, \cite[Theorem 5.2]{MR2643399}, or \cite[Theorem 2.8]{MR3618122} for a proof. As a consequence, we have
\begin{equation*}
\begin{aligned}
 \sup_{\bv \in \bW_0^{1,p}(\omega,\Omega)} \frac{b(\bv, q)}{\| \nabla \bv \|_{\bL^p(\omega,\Omega)}} & \geq \frac{b(\bw, q)}{\| \nabla \bw \|_{\bL^p(\omega,\Omega)}} =  \frac{\| q \|^{p'}_{ L^{p'}(\omega',\Omega) } }{\| \nabla \bw \|_{\bL^p(\omega,\Omega)}}
 \\
 & \gtrsim  \frac{\| q \|^{p'}_{ L^{p'}(\omega',\Omega) } }{\| r \|_{L^p(\omega,\Omega)}} \gtrsim  \| q \|_{ L^{p'}(\omega',\Omega)}.
 \end{aligned}
\end{equation*}
As we intended to show.
\end{proof}

\subsection{The mini element}
\label{sub:minipair}

This pair is considered in \cite{MR799997}, \cite[Section 4.2.4]{MR2050138} for the unweighted case and it is defined by:
\begin{align}
\label{eq:mini_V}
\V_h & = \left\{  \bv_h \in \mathbf{C}(\bar \Omega):\ \forall T \in \T, \bv_h|_{T} \in [\mathbb{P}_1(T) \oplus \mathbb{B}(T)]^d \right\} \cap \bH_0^1(\Omega),
\\
\label{eq:mini_P}
\P_h & = \left\{  q_h \in  L^2(\Omega)/\R \cap C(\bar \Omega):\ \forall T \in \T, q_h|_{T} \in \mathbb{P}_1(T) \right\},
\end{align}
where $\mathbb{B}(T)$ denotes the space spanned by local bubble functions.

We must immediately note that, for $d \in \{2,3\}$, assumption \eqref{eq:bG-bGT} is proved in \cite[Theorem 12]{MR3422453} and \cite[Theorem 8.1]{MR2121575}. Thus, we focus on the weighted LBB condition \eqref{eq:infsup_discrete}. This will be obtained with the aid of the, auxiliary, continuous inf--sup condition \eqref{eq:InfSupContinuous}. 

\begin{theorem}[discrete inf--sup condition]
\label{thm:infsup_discreteb}
Let $p \in (1,\infty)$ and $\omega \in A_p$. If $\V_h$ and $\P_h$ are defined by \eqref{eq:mini_V} and \eqref{eq:mini_P}, respectively, then we have that
\begin{equation}
\label{eq:infsup_discreteb}
\beta  \| q_h \|_{L^{p'}(\omega',\Omega)} \leq \sup_{ \bv_h \in \V_h } \frac{b(\bv_h, q_h)}{ \| \nabla \bv_h \|_{\bL^p(\omega,\Omega)}  } \quad \forall q_h \in \P_h,
\end{equation}
where the hidden constant is independent of $\T$.
\end{theorem}
\begin{proof}
Our argument will be based on \eqref{eq:InfSupContinuous} and the construction of a so--called Fortin operator \cite[Lemma 4.19]{MR2050138}. Given $\bv \in \mathbf{W}_0^{1,p}(\omega,\Omega)$, we will construct $\mathcal{F}_h \bv \in \V_h$ such that
\begin{equation}
\label{eq:Fortin_aux_1}
  b(\bv,q_h) = b(\mathcal{F}_h \bv,q_h) \quad \forall q_h \in \P_h, 
 \quad
 \| \nabla \mathcal{F}_h \bv \|_{\mathbf{L}^p(\omega,\Omega)} \lesssim \|  \nabla \bv \|_{\mathbf{L}^p(\omega,\Omega)},
\end{equation}
with a hidden constant independent of $h$. To accomplish this task, we first notice that, if $q_h \in \P_h$ then, for all $T \in \T$, $\nabla q_{h|T} \in \mathbb{R}^d$. Consequently, an integration by parts argument reveals that $\mathcal{F}_h \bv$ must be such that
\begin{equation}
\label{eq:Fortin_aux_2}
\int_{T} \bv \diff x = \int_{T} \mathcal{F}_h \bv\diff x  \quad \forall T \in \T.
\end{equation}

Let $\Pi_h$ denote the quasi--interpolation operator introduced in Section~\ref{subsec:quasi}. We define
\[
 \mathcal{F}_h \bv = \Pi_h \bv + \sum_{T \in \T} \sum_{i=1}^d \gamma_T^{i} \mathbf{e}_i b_T.
\]
Here, $\{\mathbf{e}_1, \cdots, \mathbf{e}_d \}$ denotes the canonical basis of $\mathbb{R}^d$, $\gamma_T^{i} \in \R$; $i \in \{1,\dots,d\}$, and $b_T$ is the bubble function associated with $T$. We thus have that the discrete function $\mathcal{F}_h\bv$ satisfies \eqref{eq:Fortin_aux_2} if 
\[
\gamma_T^{i} = \frac{\int_T (\bv - \Pi_h \bv)\diff x}{\int_T b_T \diff x}, \quad i \in \{ 1,\cdots,d\}, \quad T\in\T.
\]

It thus remains to prove the stability bound $\| \nabla \mathcal{F}_h \bv \|_{\mathbf{L}^p(\omega,\Omega)} \lesssim \|  \nabla \bv \|_{\mathbf{L}^p(\omega,\Omega)}$. Write
\[
 \| \nabla \mathcal{F}_h \bv \|_{\mathbf{L}^p(\omega,\Omega)} \leq  \| \nabla \Pi_h \bv \|_{\mathbf{L}^p(\omega,\Omega)} +  \left\| \nabla \left( \sum_{T \in \T} \sum_{i=1}^d \gamma_T^{i} \mathbf{e}_i b_T \right) \right \|_{\mathbf{L}^p(\omega,\Omega)} = \textrm{I} + \textrm{II},
\]
and notice that the local stability estimate \eqref{eq:local_stability} and the finite overlapping property of stars yield
\[
  \mathrm{I} = \| \nabla \Pi_h \bv \|_{\mathbf{L}^p(\omega,\Omega)}\lesssim \|  \nabla \bv \|_{\mathbf{L}^p(\omega,\Omega)}.
\]
To bound $\mathrm{II}$ we use the interpolation estimate \eqref{eq:interpolation_estimate_1} and properties of the bubble function to obtain
\begin{align*}
  | \gamma_T^i | & \lesssim |T|^{-1} h_T \| \nabla \bv\|_{\bL^{p}(\omega,\mathcal{S}_T)} \left(\int_T \omega' \diff x \right)^{\frac{1}{p'}} \\
    &\lesssim h_T^{1-d+d/p'}\| \nabla \bv \|_{\bL^{p}(\omega,\mathcal{S}_T)} \left(\fint_T \omega' \diff x \right)^{\frac{1}{p'}}.
\end{align*}
Consequently,
\begin{align*}
 \textrm{II} 
 & \lesssim \sum_{T \in \T} \sum_{i=1}^{d} |\gamma_T^i| \| \nabla b_T \|_{\bL^{p}(\omega,T)}
 \\
 & \lesssim \sum_{T \in \T} h_T^{1-d+d/p'}\| \nabla \bv\|_{\bL^{p}(\omega,\mathcal{S}_T)} \left(\fint_T \omega' \diff x \right)^{\frac{1}{p'}} h_T^{\frac{d}{p}-1}\left(\fint_T \omega \diff x \ \right)^{\frac{1}{p}}.
\end{align*}
Since $(1-d+d/p') + d/p - 1 = 0$ shape regularity allows us to conclude that
\[
 \textrm{II} \lesssim \sum_{T \in \T} \| \nabla \bv\|_{\bL^{p}(\omega,\mathcal{S}_T)} \left[ \left(\fint_T \omega \diff x \right)  \left( \fint_T \omega' \diff x \right)^{p-1} \right]^{\frac{1}{p}} \lesssim [\omega]^{\frac{1}{p}}_{A_p} \| \nabla \bv\|_{\bL^{p}(\omega,\Omega)},
\]
where we have used \eqref{A_pclass} and the finite overlapping property of stars. The collection of the derived estimates for $\mathrm{I}$ and $\mathrm{II}$ yield
\[
  \| \nabla \mathcal{F}_h \bv \|_{\mathbf{L}^p(\omega,\Omega)}  \lesssim (1 + [\omega]_{A_p}^{\frac{1}{p}}) \| \nabla \bv\|_{\bL^{p}(\omega,\Omega)}.
\]
The Fortin operator is thus constructed and this concludes the proof.
\end{proof}

\subsection{The Taylor Hood pair}
\label{sub:THpair}

The lowest order Taylor Hood element \cite{hood1974navier}, \cite{MR993474}, \cite[Section 4.2.5]{MR2050138} is defined by
\begin{align}
\label{eq:th_V}
\V_h & = \left\{  \bv_h \in \mathbf{C}(\bar \Omega): \ \forall T \in \T, \bv_h|_{T} \in \mathbb{P}_2(T)^d \right\} \cap \bH_0^1(\Omega),
\\
\label{eq:th_P}
\P_h & = \left\{ q_h \in L^2(\Omega)/\R \cap C(\bar \Omega):\  \forall T \in \T, q_h|_{T} \in \mathbb{P}_1(T) \right\}.
\end{align}
In two dimensions, estimate \eqref{eq:bG-bGT} for this pair is also obtained in \cite[Theorem 12]{MR3422453} and \cite[Theorem 8.1]{MR2121575}. In three dimensions, these references only show this result for certain classes of meshes.
As a consequence, we will focus on \eqref{eq:infsup_discrete}. Notice that, as in the unweighted case, the technique of proof must differ from that used in Section~\ref{sub:minipair}. We will follow the ideas of \cite[Section 3]{MR743884}; see also \cite[Section 4.2.5]{MR2050138}.

We begin with a preparatory step.

\begin{lemma}[perturbation]
\label{lem:perturbation}
Let $p \in (1,\infty)$ and $\omega \in A_p$. Assume that all $\{\T\}_{h>0}$ are such that every $T \in \T$ has at least $d$ edges in $\Omega$, and that $\V_h$ and $\P_h$ are defined as in \eqref{eq:th_V} and \eqref{eq:th_P}, respectively. Then we have that
\[
  h \| \nabla q_h \|_{L^{p'}(\omega',\Omega)} \lesssim \sup_{ \bv_h \in \V_h } \frac{ b(\bv_h,q_h) }{ \| \nabla \bv_h \|_{\bL^p(\omega,\Omega)}  } \quad \forall q_h \in \P_h,
\]
where the hidden constant does not depend on $h$.
\end{lemma}
\begin{proof}
We denote by $\Edges_h, \Vertices_h$, and $\Midpoints_h$ be the sets of interior edges, interior vertices, and interior edge midpoints, respectively, of $\T$. Let $\texte \in \Edges_h$ and we set $\btau_\texte$ to be a unit vector in the direction of $\texte$. Notice that there is a bijection between $\Edges_h$ and $\Midpoints_h$.

For $q_h \in \P_h$ we define $\bw_h \in \V_h$ as
\[
  \bw_h(\textv) = 0, \quad \forall \textv \in \Vertices_h,
\]
and
\[
  \bw_h(\textm) = - |\texte|^{p'} \btau_\texte \signum(\partial_{\btau_\texte} q_h )|\partial_{\btau_\texte} q_h|^{p'-1} \frac{\omega'(T)}{|T|}, \quad \forall \textm \in \Midpoints_h.
\]
Let $\{\phi_\textm\}_{\textm \in \Midpoints_h} \cup \{ \phi_\textv \}_{\textv \in \Vertices_h}$ be the Lagrange nodal basis for piecewise quadratics over $\T$. Upon expanding $\bw_h$ on this basis we realize that
\begin{align*}
  \| \nabla \bw_h \|_{\bL^p(\omega, \Omega)}^p &= \sum_{T \in \T} \int_T \omega \left|\sum_{\textm \in \Midpoints_h : \textm \in T} \bw_h(\textm) \nabla \phi_\textm \right|^p \diff x \\
    &\lesssim h^{p'} \sum_{T \in \T} \frac{ \omega(T) \left[\omega'(T)\right]^p}{|T|^p} \sum_{\textm \in \Midpoints_h : \textm \in T} |\partial_{\btau_\texte} q_h|^{p'} \\
    & \lesssim h^{p'} [\omega]_{A_p} \sum_{T \in \T} \omega'(T) |\nabla q_h|^{p'} \lesssim h^{p'} \| \nabla q_h\|_{\bL^{p'}(\omega',\Omega)}^{p'}.
\end{align*}

Recall now (see \cite[Tables 8.2 and 8.3]{MR2050138}) that for $d \in \{2,3\}$ there is a quadrature formula on the unit simplex which is exact for quadratics, it is supported on the vertices and edge midpoints of the simplex, and has positive weights on the midpoints. Let $\{\varrho_\textm\}$ be the weights of this formula, then we have that
\begin{align*}
  b(\bw_h,q_h) &= - \sum_{T \in \T} \nabla q_h \cdot \int_T \bw_h \diff x \\
    &= \sum_{T \in \T} \omega'(T) \nabla q_h \cdot \sum_{\textm \in \Midpoints_h : \textm \in T} \varrho_m\btau_{\texte}  |e|^{p'} \signum(\partial_{\btau_\texte} q_h )|\partial_{\btau_\texte} q_h|^{p'-1} \\
    &\gtrsim h^{p'} \sum_{T \in \T} \omega'(T) \sum_{\texte \in \Edges_h: \texte \subset T} |\partial_{\btau_\texte} q_h|^{p'} 
    \gtrsim h^{p'} \sum_{T \in \T} \omega'(T) |\nabla q_h|^{p'},
\end{align*}
where, in the last step, we used that the mesh assumption implies that for any element $T$ the collection $\{\btau_\texte\}_{\texte \in \Edges_h: \texte \subset T}$ spans $\R^d$. Conclude by recalling that $\nabla q_h$ is constant over $T$.
\end{proof}

With this result at hand we now prove \eqref{eq:infsup_discrete} for the Taylor Hood pair.

\begin{theorem}[discrete inf--sup condition]
In the setting of Lemma~\ref{lem:perturbation}, we have
\begin{equation}
\label{eq:infsup_discretebTH}
  \| q_h \|_{L^{p'}(\omega',\Omega)} \lesssim \sup_{ \bv_h \in \V_h } \frac{ b(\bv_h, q_h) }{ \| \nabla \bv_h \|_{\bL^p(\omega,\Omega)}  }  \quad \forall q_h \in \P_h,
\end{equation}
where the hidden constant is independent of $h$.
\label{lm:infsup_discretebTH}
\end{theorem}
\begin{proof}
Given $q_h \in \P_h \subset L^{p'}(\omega',\Omega)/\R$, let $\bw_{q_h} \in \bW^{1,p}_0(\omega,\Omega)$ be the function constructed in the course of the proof of \eqref{eq:InfSupContinuous} and $\Pi_h$ the interpolant, onto $\V_h$, described in Section~\ref{subsec:quasi}. The properties of $\Pi_h$ and arguing as in the proof of \eqref{eq:InfSupContinuous} show that
\begin{align*}
  \sup_{ \bv_h \in \V_h } \frac{ b(\bv_h, q_h)}{ \| \nabla \bv_h \|_{\bL^p(\omega,\Omega)}  }  &\geq 
    \frac{ b(\Pi_h \bw_{q_h} , q_h) }{ \| \nabla \Pi_h \bw_{q_h} \|_{\bL^p(\omega,\Omega)}  } \\
    &\gtrsim \| q_h \|_{L^{p'}(\omega',\Omega)} +
    \frac{b\left( \Pi_h \bw_{q_h} - \bw_{q_h}, q_h \right)}{ \| \nabla \bw_{q_h} \|_{\bL^p(\omega,\Omega)}  }.
\end{align*}

Integration by parts, and the properties of $\Pi_h$ show that
\begin{align*}
  \frac{b\left( \Pi_h \bw_{q_h} - \bw_{q_h}, q_h \right)}{ \| \nabla \bw_{q_h} \|_{\bL^p(\omega,\Omega)}  } &\geq
  - \frac{ \| \nabla q_h \|_{L^{p'}(\omega',\Omega)} \| \bw_{q_h} - \Pi_h \bw_{q_h} \|_{\bL^p(\omega,\Omega)}  }{ \| \nabla \bw_{q_h} \|_{\bL^p(\omega,\Omega)}  } \\
  &\gtrsim -h \| \nabla q_h \|_{L^{p'}(\omega',\Omega)}.
\end{align*}
Lemma~\ref{lem:perturbation} allows us to conclude.
\end{proof}


\begin{thebibliography}{10}

\bibitem{MR3618122}
G.~Acosta and R.~G. Dur\'{a}n, \emph{Divergence operator and related
  inequalities}, SpringerBriefs in Mathematics, Springer, New York, 2017.
  \MR{3618122}

\bibitem{MR3215609}
H.~Aimar, M.~Carena, R.~Dur\'an, and M.~Toschi, \emph{Powers of distances to
  lower dimensional sets as {M}uckenhoupt weights}, Acta Math. Hungar.
  \textbf{143} (2014), no.~1, 119--137. \MR{3215609}

\bibitem{MR799997}
D.~N. Arnold, F.~Brezzi, and M.~Fortin, \emph{A stable finite element for the
  {S}tokes equations}, Calcolo \textbf{21} (1984), no.~4, 337--344 (1985).
  \MR{799997}

\bibitem{MR972452}
C.~Bernardi, C.~Canuto, and Y.~Maday, \emph{Generalized inf-sup conditions for
  {C}hebyshev spectral approximation of the {S}tokes problem}, SIAM J. Numer.
  Anal. \textbf{25} (1988), no.~6, 1237--1271. \MR{972452}

\bibitem{MR2373954}
S.~C. Brenner and L.~R. Scott, \emph{The mathematical theory of finite element
  methods}, third ed., Texts in Applied Mathematics, vol.~15, Springer, New
  York, 2008. \MR{2373954}

\bibitem{MR3582412}
M.~Bul{\'\i}{\v c}ek, J.~Burczak, and S.~Schwarzacher, \emph{A unified theory
  for some non-{N}ewtonian fluids under singular forcing}, SIAM J. Math. Anal.
  \textbf{48} (2016), no.~6, 4241--4267. \MR{3582412}

\bibitem{MR805809}
S.~Chanillo and R.~L. Wheeden, \emph{Weighted {P}oincar\'e and {S}obolev
  inequalities and estimates for weighted {P}eano maximal functions}, Amer. J.
  Math. \textbf{107} (1985), no.~5, 1191--1226. \MR{805809}

\bibitem{MR2217368}
H.~Chen, \emph{Pointwise error estimates for finite element solutions of the
  {S}tokes problem}, SIAM J. Numer. Anal. \textbf{44} (2006), no.~1, 1--28.
  \MR{2217368}

\bibitem{CiarletBook}
P.~G. Ciarlet, \emph{The finite element method for elliptic problems}, SIAM,
  Philadelphia, PA, 2002. \MR{1930132}

\bibitem{MR0400739}
Ph. Cl\'ement, \emph{Approximation by finite element functions using local
  regularization}, Rev. Fran\c{c}aise Automat. Informat. Recherche
  Op\'erationnelle S\'er. \textbf{9} (1975), no.~{\rm R}-2, 77--84.
  \MR{0400739}

\bibitem{MR2643399}
L.~Diening, M.~R\r{u}\v{z}i\v{c}ka, and K.~Schumacher, \emph{A decomposition
  technique for {J}ohn domains}, Ann. Acad. Sci. Fenn. Math. \textbf{35}
  (2010), no.~1, 87--114. \MR{2643399}

\bibitem{MR1800316}
J.~Duoandikoetxea, \emph{Fourier analysis}, Graduate Studies in Mathematics,
  vol.~29, American Mathematical Society, Providence, RI, 2001, Translated and
  revised from the 1995 Spanish original by David Cruz-Uribe. \MR{1800316}

\bibitem{MR935076}
R.~Dur\'{a}n, R.~H. Nochetto, and J.~P. Wang, \emph{Sharp maximum norm error
  estimates for finite element approximations of the {S}tokes problem in
  {$2$}-{D}}, Math. Comp. \textbf{51} (1988), no.~184, 491--506. \MR{935076}

\bibitem{DDO:17}
R.~G. Dur\'an, I.~Drelichman, and I.~Ojea, \emph{A weighted setting for the the
  numerical approximation of the {P}oisson problem with singular sources},
  (2018), ArXiv: arXiv:1809.03529.

\bibitem{MR2164092}
R.~G. Dur\'an and A.~L. Lombardi, \emph{Error estimates on anisotropic {$Q_1$}
  elements for functions in weighted {S}obolev spaces}, Math. Comp. \textbf{74}
  (2005), no.~252, 1679--1706. \MR{2164092}

\bibitem{MR2731700}
R.~G. Dur\'{a}n and F.~L\'{o}pez~Garc\'{i}a, \emph{Solutions of the divergence
  and {K}orn inequalities on domains with an external cusp}, Ann. Acad. Sci.
  Fenn. Math. \textbf{35} (2010), no.~2, 421--438. \MR{2731700}

\bibitem{MR1880723}
R.~G. Dur\'{a}n and M.~A. Muschietti, \emph{An explicit right inverse of the
  divergence operator which is continuous in weighted norms}, Studia Math.
  \textbf{148} (2001), no.~3, 207--219. \MR{1880723}

\bibitem{MR995211}
R.~G. Dur\'{a}n and R.~H. Nochetto, \emph{Weighted inf-sup condition and
  pointwise error estimates for the {S}tokes problem}, Math. Comp. \textbf{54}
  (1990), no.~189, 63--79. \MR{995211}

\bibitem{MR2050138}
A.~Ern and J.-L. Guermond, \emph{Theory and practice of finite elements},
  Applied Mathematical Sciences, vol. 159, Springer-Verlag, New York, 2004.
  \MR{2050138}

\bibitem{MR643158}
E.~B. Fabes, C.~E. Kenig, and R.~P. Serapioni, \emph{The local regularity of
  solutions of degenerate elliptic equations}, Comm. Partial Differential
  Equations \textbf{7} (1982), no.~1, 77--116. \MR{643158}

\bibitem{MR1601373}
R.~Farwig and H.~Sohr, \emph{Weighted {$L^q$}-theory for the {S}tokes resolvent
  in exterior domains}, J. Math. Soc. Japan \textbf{49} (1997), no.~2,
  251--288. \MR{1601373}

\bibitem{MR1292572}
B.~Franchi, C.~E. Guti\'{e}rrez, and R.~L. Wheeden, \emph{Two-weight
  {S}obolev-{P}oincar\'{e} inequalities and {H}arnack inequality for a class of
  degenerate elliptic operators}, Atti Accad. Naz. Lincei Cl. Sci. Fis. Mat.
  Natur. Rend. Lincei (9) Mat. Appl. \textbf{5} (1994), no.~2, 167--175.
  \MR{1292572}

\bibitem{MR2808162}
G.~P. Galdi, \emph{An introduction to the mathematical theory of the
  {N}avier-{S}tokes equations}, second ed., Springer Monographs in Mathematics,
  Springer, New York, 2011. \MR{2808162}

\bibitem{MR3422453}
V.~Girault, R.~H. Nochetto, and L.~R. Scott, \emph{Max-norm estimates for
  {S}tokes and {N}avier-{S}tokes approximations in convex polyhedra}, Numer.
  Math. \textbf{131} (2015), no.~4, 771--822. \MR{3422453}

\bibitem{MR2121575}
V.~Girault, R.~H. Nochetto, and R.~Scott, \emph{Maximum-norm stability of the
  finite element {S}tokes projection}, J. Math. Pures Appl. (9) \textbf{84}
  (2005), no.~3, 279--330. \MR{2121575}

\bibitem{MR851383}
V.~Girault and P.-A. Raviart, \emph{Finite element methods for
  {N}avier-{S}tokes equations}, Springer Series in Computational Mathematics,
  vol.~5, Springer-Verlag, Berlin, 1986, Theory and algorithms. \MR{851383}

\bibitem{MR2491902}
V.~Gol'dshtein and A.~Ukhlov, \emph{Weighted {S}obolev spaces and embedding
  theorems}, Trans. Amer. Math. Soc. \textbf{361} (2009), no.~7, 3829--3850.
  \MR{2491902}

\bibitem{MR2305115}
J.~Heinonen, T.~Kilpel\"{a}inen, and O.~Martio, \emph{Nonlinear potential
  theory of degenerate elliptic equations}, Dover Publications, Inc., Mineola,
  NY, 2006, Unabridged republication of the 1993 original. \MR{2305115}

\bibitem{hood1974navier}
P~Hood and C~Taylor, \emph{Navier-{S}tokes equations using mixed
  interpolation}, Finite element methods in flow problems (1974), 121--132.

\bibitem{MR2454455}
F.~I. Mamedov and R.~A. Amanov, \emph{On some nonuniform cases of weighted
  {S}obolev and {P}oincar\'{e} inequalities}, Algebra i Analiz \textbf{20}
  (2008), no.~3, 163--186. \MR{2454455}

\bibitem{MR2987056}
M.~Mitrea and M.~Wright, \emph{Boundary value problems for the {S}tokes system
  in arbitrary {L}ipschitz domains}, Ast\'erisque (2012), no.~344, viii+241.
  \MR{2987056}

\bibitem{MR0293384}
B.~Muckenhoupt, \emph{Weighted norm inequalities for the {H}ardy maximal
  function}, Trans. Amer. Math. Soc. \textbf{165} (1972), 207--226.
  \MR{0293384}

\bibitem{MR0474884}
F.~Natterer, \emph{\"{U}ber die punktweise {K}onvergenz finiter {E}lemente},
  Numer. Math. \textbf{25} (1975/76), no.~1, 67--77. \MR{0474884}

\bibitem{NOS3}
R.~H. Nochetto, E.~Ot\'arola, and A.~J. Salgado, \emph{Piecewise polynomial
  interpolation in {M}uckenhoupt weighted {S}obolev spaces and applications},
  Numer. Math. \textbf{132} (2016), no.~1, 85--130. \MR{3439216}

\bibitem{OS:17infsup}
E.~Ot\'arola and A.~J. Salgado, \emph{The {P}oisson and {S}tokes problems on
  weighted spaces in {L}ipschitz domains and under singular forcing}, J. Math.
  Anal. Appl. \textbf{471} (2019), no.~1, 599--612.

\bibitem{MR645661}
R.~Rannacher and R.~Scott, \emph{Some optimal error estimates for piecewise
  linear finite element approximations}, Math. Comp. \textbf{38} (1982),
  no.~158, 437--445. \MR{645661}

\bibitem{MR2548872}
K.~Schumacher, \emph{Solutions to the equation {${\rm div}\,u=f$} in weighted
  {S}obolev spaces}, Parabolic and {N}avier-{S}tokes equations. {P}art 2,
  Banach Center Publ., vol.~81, Polish Acad. Sci. Inst. Math., Warsaw, 2008,
  pp.~433--440. \MR{2548872}

\bibitem{MR1011446}
L.~Ridgway Scott and Shangyou Zhang, \emph{Finite element interpolation of
  nonsmooth functions satisfying boundary conditions}, Math. Comp. \textbf{54}
  (1990), no.~190, 483--493. \MR{1011446}

\bibitem{MR0337032}
R.~Scott, \emph{Finite element convergence for singular data}, Numer. Math.
  \textbf{21} (1973/74), 317--327. \MR{0337032}

\bibitem{MR1774162}
B.~O. Turesson, \emph{Nonlinear potential theory and weighted {S}obolev
  spaces}, Lecture Notes in Mathematics, vol. 1736, Springer-Verlag, Berlin,
  2000. \MR{1774162}

\bibitem{MR743884}
R.~Verf\"{u}rth, \emph{Error estimates for a mixed finite element approximation
  of the {S}tokes equations}, RAIRO Anal. Num\'{e}r. \textbf{18} (1984), no.~2,
  175--182. \MR{743884}

\bibitem{MR993474}
R.~Verf\"urth, \emph{A posteriori error estimators for the {S}tokes equations},
  Numer. Math. \textbf{55} (1989), no.~3, 309--325. \MR{993474}

\end{thebibliography}

\providecommand{\bysame}{\leavevmode\hbox to3em{\hrulefill}\thinspace}
\providecommand{\MR}{\relax\ifhmode\unskip\space\fi MR }
\providecommand{\MRhref}[2]{%
  \href{http://www.ams.org/mathscinet-getitem?mr=#1}{#2}
}
\providecommand{\href}[2]{#2}

\end{document}